\documentclass[12pt, oneside]{article}
\usepackage[english]{babel}
\usepackage[T1]{fontenc}

\usepackage{amssymb}
\usepackage{amsmath}
\usepackage{hhline}
\usepackage{longtable}
\usepackage{amscd}
\usepackage{theorem}
\usepackage{array}
\usepackage{delarray}
\usepackage{multicol}
\usepackage{makeidx}
\usepackage{soul}
\usepackage{float}

\usepackage{tikz}
\usepackage{calc}
\usetikzlibrary{angles,quotes}
\usetikzlibrary{calc}
\usetikzlibrary{patterns}
\usetikzlibrary{graphs}
\usetikzlibrary{graphs.standard}
\usetikzlibrary{decorations.markings,shapes.geometric,positioning}

\newtheorem{theorem}{Theorem}[section]

\newtheorem{conjecture}{Conjecture}[section]

\newtheorem{remark}{Remark}[section]
\newtheorem{claim}{Claim}[theorem]

\newtheorem{problem}{Problem}
\newtheorem{proposition}{Proposition}[section]

\newenvironment{proof}[1][Proof]{\textbf{#1.} }{\ \rule{0.5em}{0.5em}}

\newcommand{\dist}{\mathrm{dist}}

\usepackage{tikz}

\usepackage{caption}

\author{Maidoun Mortada\footnote{KALMA Laboratory, Faculty of Sciences, Lebanese University, Baalbek, Lebanon.}, Olivier Togni\footnote{LIB Laboratory, Université Bourgogne Europe, Dijon, France.}}
\title{On $S$-packing Coloring of Bounded Degree Graphs}
\begin{document}
\maketitle

\begin{abstract}
Given a sequence $S=(s_1,s_2,\ldots,s_p)$, $p\geq 2$, of non-decreasing integers, an $S$-packing coloring of a graph $G$ is a partition of its vertex set into $p$ disjoint sets $V_1,\ldots, V_p$ such that any two distinct vertices of $V_i$ are at a distance greater than $s_i$, $1\le i\le p$. In this paper, we study the $S$-packing coloring problem on graphs of bounded maximum degree and for sequences mainly containing 1's and 2's ($i^r$ in a sequence means $i$ is repeated $r$ times). Generalizing existing results for subcubic graphs, we prove a series of results on graphs of maximum degree $k$: We show that graphs of maximum degree $k$ are $(1^{k-1},2^k)$-packing colorable. Moreover, we refine this result for restricted subclasses: A graph of maximum degree $k$ is said to be $t$-saturated, $0\le t\le k$, if every vertex of degree $k$ is adjacent to at most $t$ vertices of degree $k$. We prove that any graph of maximum degree $k\ge 3$ is $(1^{k-1}, 3)$-packing colorable if it is 0-saturated, $(1^{k-1}, 2)$-packing colorable if it is $t$-saturated, $1\leq t\leq k-2$; and $(1^{k-1},2^{k-1})$-packing colorable if it is $(k-1)$-saturated. We also propose some conjectures and questions.
\end{abstract}

\section{Introduction}

For the classical vertex coloring problem, the famous Brooks' Theorem asserts that a connected graph of maximum degree $k$ can be colored using at most $k$ colors unless it is an odd cycle or a complete graph. A generalization of coloring is $d$-distant coloring, in which two vertices can receive the same color if they are at a distance of at least $d+1$ apart. Equivalently, it corresponds to a coloring of the $d^{th}$ power $G^d$ of the graph $G$. A Brooks-type upper bound for the minimum number of colors of a 2-distant coloring of a graph of maximum degree $k$ is $k^2+1$. It can be obtained by a greedy coloring algorithm, and it is tight for Moore graphs like the Petersen graph, which has 10 vertices at a pairwise distance at most 2 and thus needs 10 colors for a 2-distant coloring.

As a generalization of $d$-distant coloring, $S$-packing coloring allows us to define, for each color used, the minimum distance of reutilization (or in which power of $G$ the conflicts are considered for this color). Formally, for a sequence of non-decreasing integers $S=(s_1,s_2,\dots, s_p)$, an $S$-packing coloring of a graph $G$ is a partition of its vertex set $V$ into $p$ sets $V_1, V_2,\dots, V_p$ such that any two distinct vertices of $V_i$ are at distance at least $s_i+1$, $1\le i\le p$. A $d$-distant coloring is thus a $(d,d,\dots,d)$-packing coloring. The notion of packing and $S$-packing coloring was introduced by Goddard et al. in 2008~\cite{GHH+}. The interested reader is referred to the paper of Bre\v{s}ar, Ferme and Klav\v{z}ar~\cite{BFK} for a complete survey up to 2020.

If the sequence $S$ contains only 1's and 2's, then we are in a type of coloring that can be viewed as intermediate between proper and 2-distant coloring. We denote by $(1^k,2^\ell)$ the sequence of length $k+\ell$ with $k$ ones and $\ell$ twos. 
Brooks' Theorem can be rewritten as any connected graph of maximum degree $k$ is $(1^k)$-packing colorable except if it is an odd cycle or a complete graph. For 2-distant coloring, we have that any graph of maximum degree $\ell$ is $(2^{\ell^2})$-packing colorable except if $G^2$ is an odd cycle or a complete graph. One can thus ask what happens if both $k$ and $\ell$ are non-zero.

For graphs of maximum degree $\Delta =2$, i.e., union of paths and cycles, it is easy to characterize the minimal values of $k,\ell$ for which they are $(1^k,2^\ell)$-colorable: any path is $(1^2)$-, $(1,2^2)$- and $(2^3)$-packing colorable. Any cycle except $C_5$ is $(1^2,2)$-packing colorable. In addition, any cycle except $C_5$ is $(1,2^2)$- and $(2^4)$-packing colorable.

In contrast, for maximum degree 3, things are much more complicated. Finding bounds for $k,\ell$ for which any subcubic graph is $(1^k,2^\ell)$-packing colorable was the topic of many papers in recent years~\cite{GT, BKKRW, LLRY, YW, M24, TT, MT, MT2}. The main open problems are whether all cubic graphs (except the Petersen graph $P$) are $(1,2^5)$-packing colorable and $(1,1,2,3)$-packing colorable. To approach these problems, many subclasses of subcubic graphs have been considered, and in particular, the $i$-saturated subclasses (that can be extended to any maximum degree graphs): We call {\em $k$-degree graph} every graph $G$ with $\Delta(G)\leq k$. In a graph, a vertex of degree $i$ is called an {\em $i$-vertex}. A $k$-degree graph is said to be {\em $i$-saturated}, $0\le i\le k$, if each $k$-vertex is adjacent to at most $i$ $k$-vertices. See~\cite{MT3} and references therein for recent results on $S$-packing coloring of $i$-saturated subcubic graphs.

Note that for maximum degree larger than 3, there are only results for restricted subclasses of quartic graphs like distance graphs~\cite{BFK2, HH}.

\paragraph{Notation:}
For a graph $G$ and a vertex $x$ in $G$, let $N_G(x)$ be the set of neighbors of $x$ in $G$ and let $d_G(x)=|N(x)|$ be the degree of $x$ in $G$. If no confusion exists, $N(x)$ (resp $d(x)$) is used instead of $N_G(x)$ (resp. $d_G(x)$). For $X\subseteq V(G)$, $N_G(X)$ (or simply $N(X)$) is the set of neighbors of the vertices in $X$. The distance $\dist_G(x,y)$ (or simply $\dist(x,y)$) between two vertices $x,y$ in $G$ is the length of a shortest path between $x$ and $y$. For a path $P=x_1 \dots x_n$, we call $x_1$ and $x_n$ the {\em ends} of $P$, while each other vertex is called an {\em interior} vertex. The {\em length} of a path $P$, $l(P)$, is the number of its edges.  A path $P$ in a graph $G$ is said to be {\em maximal} if $P$ is not a subpath of any other path in $G$. 

Let $G$ be a $k$-degree graph. If $H$ is a subgraph of $G$ or a subset of $V(G)$ and if $x$ is a vertex in $H$, by saying $x$ is an $i$-vertex, we mean that $x$ is an $i$-vertex in $G$, $0\leq i\leq k$. That is, $x$ may not have $i$ neighbors in $H$ but has them in $G$. This is applied to the whole paper. 
Moreover, we will use $1_1,\ldots, 1_k$ (or simply $1$ if $k=1$) and $2_1,\ldots, 2_{\ell}$ (or simply $2$ if $\ell=1)$ to denote the colors of a $(1^k,2^{\ell})$-packing coloring.

For a graph $G$, the $k^{th}$-power graph $G^k$ of $G$ is the graph obtained from $G$ after adding an edge between any two non-adjacent vertices $x$ and $y$ in $G$ with $\dist_G(x,y)\leq k$.  It is important to note that if $G^2$ is $k$-colorable, $k\geq 2$, then $G$ is $(2^k)$-packing colorable. In fact, if two vertices $x$ and $y$ in $G^2$ are of the same color $1_i$ for some $i$, $1\leq i\leq k$, when considering  a proper $k$-coloring of $G^2$ using the  colors $1_1,\ldots, 1_k$, then $x$ and $y$ are not adjacent in $G^2$, and so $\dist_G(x,y)>2$. Thus color each vertex $x$ in $G$ by $2_i$ if $x$ is of color $1_i$ when considering a $(1^k)$-packing coloring of $G^2$, and so we obtain a $(2^k)$-packing coloring of $G$.

\paragraph{Summary:}
Our aim in this paper is to extend the results on $(1^k,2^\ell)$-packing colorability of graphs from maximum degree $3$ to any $k$. In particular, we consider $i$-saturated $k$-degree graphs. In Section 2, we prove that every 0-saturated $k$-degree graph, $k\geq 3$,  is $(1^{k-1}, 3)$-packing colorable while in Section 3, that every $t$-saturated $k$-degree graph, $k\ge 3$ and $1\leq t\leq k-2$, is $(1^{k-1}, 2)$-packing colorable. In Section 4, we prove that  every $(k-1)$-saturated $k$-degree graph, $k\geq 4$, is $(1^{k-1},2^{k-1})$-packing colorable. Finally, it is proven in Section 5 that any $k$-degree graph is $(1^{k-1},2^k)$-packing colorable, and we discuss in Section 6 the sharpness of these results and propose some open problems.

\section{0-Saturated $k$-Degree Graphs }

In this section we consider $0$-saturated $k$-degree graphs and extend the known results for subcubic graphs~\cite{YW,M24} to any maximum degree $k\ge 3$.
\begin{theorem}
    Let $G$ be a 0-saturated $k$-degree graph, $k\geq 3$,  then $G$ is $(1^{k-1}, 3)$-packing colorable.
\end{theorem}
\begin{proof}
     Let $T$ be an independent set of $G$, $X(T)$ be the set of $k$-vertices in $T$ and let $Y(T)$ be the set of $i$-vertices in $T$, $i< k$. We define $\phi(T)=|X(T)|+0.6|Y(T)|$ and we denote by $\overline{T}$ the set $V(G)\setminus T$.
     
     Let $S$ be an independent set of $G$ such that $\phi(S)\geq \phi(T)$ for every independent set $T$. By the maximality of $\phi(S)$, it is obvious to notice that every vertex in $\overline{S}$ has a neighbor in $S$. Moreover, if $u$ is a $k$-vertex in $\overline{S}$, then $d_{G[\overline{S}]}(u)\leq k-2$. In fact, suppose that $d_{G[\overline{S}]}(u)= k-1$ and let $v$ be the unique neighbor of $u$ in $S$. As $G$ is 0-saturated $k$-degree graph, then $v$ is an $i$-vertex with $i< k$. Consequently, $S'=(S\setminus \{v\})\cup \{u\}$ is an independent set with $\phi(S')=\phi(S)-0.6+1$, a contradiction. Therefore, $\Delta(G[\overline{S}])\leq k-2$.
     
       Color the vertices of $S$ by $1_{k-1}$. Our plan is to prove that the vertices of $\overline{S}$ can be colored by the colors $1_1,\dots, 1_{k-2}$ and the color 3. 
       
      Remark that if $K$ is a complete subgraph of $G[\overline{S}]$ of order $k-1$, then $K$ contains at most one $k$-vertex of $G$ since $G$ is 0-saturated. Moreover, if $K_1$ and $K_2$ are two distinct complete subgraphs of $G[\overline{S}]$ of order $k-1$, then  $V(K_1)\cap V(K_2)=\emptyset$. In fact, if there exists $u\in V(K_1)\cap V(K_2)$, then $u$ has $k-2$ neighbors in $K_1$ and at least one neighbor in $K_2\setminus K_1$, and so $d_{G[\overline{S}]}(u)\geq  k-1$, a contradiction. 
      
     Let $C\subset \overline{S}$ such that
        for every  complete subgraph $K$ of order $k-1$  in $G[\overline{S}]$, we have  $|V(K)\cap C|=1$ and if $u\in V(K)\cap C$ then $u$ is a non $k$-vertex. 
   Hence every connected component of $G[\overline{S}\setminus C]$ has no complete subgraph of order $k-1$. This means that, by Brooks' Theorem,  every connected component of    $G[\overline{S}\setminus C]$ is $(k-2)$-colorable. Thus consider a coloring of each connected component of $G[\overline{S}]\setminus C$ using the colors $1_1,\dots, 1_{k-2}$. The next step is to prove that the distance between any two vertices in $C$ is at least four, so that we can then color the vertices of $C$ by 3.  
   
    Suppose to the contrary that there exist $u,v\in C$ such that  $d(u,v)<4$. Let $K_1$ ( resp. $K_2$) be the complete subgraph of order $k-1$ containing $u$ (resp $v$) in $G[\overline{S}]$. As $u$ (resp. $v$) is a non $k$-vertex and since $u$ (resp. $v$) has $(k-2)$ neighbors in $K_1$ (resp. $K_2$), then $N_{G[\overline{S}]}(u)\subset K_1$ (resp. $N_{G[\overline{S}]}(v)\subset K_2$), and $u$ (resp. $v$) has a unique neighbor in $S$. Consequently,  $u$ and $v$ are not adjacent.   Moreover, since $V(K_1)\cap V(K_2)=\emptyset$, then $u$ and $v$ have no common neighbor in $\overline{S}$. 
        We will prove now that $u$ and $v$ have no common neighbor in $S$. Suppose $u$ and $v$ have a common neighbor in $S$, say $x$. Then  $x$ is the unique neighbor of $u$ (resp. $v$) in $S$.  Regardless of the nature of $x$, we have  $S'=(S\setminus\{x\})\cup \{u,v\}$ is an independent set with $\phi(S')\geq \phi(S)-1+1.2$, a contradiction. 
    We still need to prove that a neighbor of $u$ is not adjacent to a neighbor of $v$.  Suppose to the contrary that there exists a neighbor of $u$, say $u'$, and a neighbor of $v$, say $v'$, such that $u'$ and $v'$ are adjacent. Remark that $u'$ and $v'$ cannot be both in $S$ since they are adjacent.   Without loss of generality, suppose $u'\in \overline{S}$. Then $u'\in V(K_1)$ since $N_{G[\overline{S}]}(u)\subset K_1$.  We will study first the case when  $v'$ is in $\overline{S}$.  In this case, we get  $v'\in V(K_2)$ since $N_{G[\overline{S}]}(v)\subset K_2$. Consequently, since $u'$ (resp. $v'$) has $(k-2)$ neighbors in $K_1$ (resp. $K_2$) and a neighbor in $S$, we get $u'$ and $v'$ are both $k$-vertices, a contradiction.  We still need to study the case when $v'$ is in $S$. Recall that, in this case, $v'$ is the unique neighbor of $v$ in $S$. Moreover, if $u'$ is a $k$-vertex, then a neighbor of $u'$ in  $S$ is not a  $k$-vertex since $G$ is 0-saturated. Hence $S'=(S\setminus(N(u'))\cup\{u',v\} $ is an independent set with $\phi(S')\geq \phi(S)-1+1.2$ if $u'$ is a non $k$-vertex and $\phi(S')=\phi(S)-1.2+1.6$ otherwise, a contradiction. Thus any two vertices in $C$ are at distance at least 4 from each other. Color each vertex of $C$ by 3, and so we obtain a $(1^{k-1},3)$-packing coloring of $G$.\end{proof}

\section{$t$-Saturated $k$-Degree Graphs with $1\leq t\leq k-2$}

Using an iterated version of the method of the maximum independent set used in Section 2, we prove the following result that extends the one for subcubic graphs~\cite{MT} to any maximum degree $k\ge 3$.
\begin{theorem}\label{tsat}
    Let $G$ be a $t$-saturated $k$-degree graph, $k\geq 3$ and $1\leq t\leq k-2$,  then $G$ is $(1^{k-1},2)$-packing colorable. 
\end{theorem}
\begin{proof}
Without loss of generality, suppose $G$ is connected and that $t=k-2$.
Let $T$ be an independent set in $G$. We denote by   $X_1(T)$  the set of $k$-vertices  in $T$ and by $Y_1(T)$  the set of non $k$-vertices in $T$. We define $\phi_1(T)=|X_1(T)|+0.6|Y_1(T)|$.
 
 Let $S_1$ be an independent set such that $\phi_1(S_1)\geq \phi_1(T)$ for every independent set $T$. 
 Clearly, by the maximality of $\phi_1(S_1)$, each vertex in $G\setminus S_1$ has a neighbor in $S_1$. 
   We are going to introduce  a sequence of independent sets $S_2$,\dots, $S_{k-2}$ in $G\setminus S_1$, such that $S_i$ is defined starting from $i=2$ as follows: Let $T$ be an independent set in $G[V(G)\setminus(S_1\cup\dots\cup S_{i-1}) ]$. We denote by   $X_i(T)$  the set of $k$-vertices  in $T$ and by $Y_i(T)$  the set of non $k$-vertices in $T$. We define $\phi_i(T)=|X_i(T)|+0.6|Y_i(T)|$.  $S_i$ is defined to be an independent set in $G[V(G)\setminus(S_1\cup\dots \cup S_{i-1}) ]$ such that  
 $\phi_i(S_i)\geq \phi_i(T)$ for every independent set $T$ in $G[V(G)\setminus(S_1\cup\dots \cup S_{i-1}) ]$. 

Let $S= S_1\cup \dots \cup S_{k-2}$  and let $\overline{S}= V(G)\setminus (S_1\cup \dots \cup S_{k-2})$. We have the following observation:
\begin{remark}\label{r3.1}
\begin{enumerate}
\item By the maximality of $\phi_j(S_j)$, each vertex in $G[V(G)\setminus(S_1\cup\dots \cup  S_{i}) ]$ has a neighbor in $S_j$ for every $j$,  $1\leq j\leq i$ and $1\leq i\leq k-2$.

\item By (1), we can deduce that each $m$-vertex in $\overline{S}$, $m\leq k-2$, has no neighbor in $\overline{S}$, each $(k-1)$-vertex in $\overline{S}$ has at most one neighbor in $\overline{S}$ and each $k$-vertex in $\overline{S}$ has at most two neighbors in $\overline{S}$. \end{enumerate}
\end{remark}
For a $k$-vertex in $\overline{S}$, we have the following result:
\begin{claim}\label{c3.1.1}
 Let $x$ be a $k$-vertex in $\overline{S}$, then:\begin{enumerate}
     \item If $x$ has a unique neighbor, say $u$,  in $S_i$ for some $i$, $1\leq i\leq k-2$, then $u$ is a $k$-vertex. 
     \item If $x$ has two neighbors in $\overline{S}$, then each neighbor of $x$ in $\overline{S}$ is a non $k$-vertex.
 \end{enumerate}
 \end{claim}
 \begin{proof}
     \begin{enumerate}
     \item Suppose to the contrary that $u$ is a non $k$-vertex, then  $S'_i=(S_i\setminus\{u\})\cup\{x\}$ is an independent set in $G[V(G)\setminus(S_1\cup\dots\cup S_{i-1}) ]$ with $\phi_i(S'_i)= \phi_i(S_i)-0.6+1$, a contradiction.  
    \item  Since $x$ has two neighbors in  $\overline{S}$, then $x$ has a unique neighbor in $S_i$ for every $i$, $1\leq i\leq k-2$.  By (1), the neighbor of $x$ in $S_i$ is a $k$-vertex for every $i$, $1\leq i\leq k-2$. Thus each neighbor of $x$ in $\overline{S}$ is a non $k$-vertex since $G$ is $(k-2)$-saturated. \end{enumerate} 
 \end{proof}

 By Claim~\ref{c3.1.1} and Remark~\ref{r3.1}, one can deduce that $G[\overline{S}]$ contains no cycle and each path in $G[\overline{S}]$ is a path of length at most two. Moreover, we can distinguish four types of maximal paths of length at least one in $G[\overline{S}]$ (see Figure~\ref{fig:types}):
\begin{enumerate}
\item A maximal path of length one whose both ends are $(k-1)$-vertices, and this type will be denoted by $\mathcal{P}_1$.
\item A maximal path of length one whose one of its ends is a $k$-vertex and the other is a $(k-1)$-vertex, and this type will be denoted by $\mathcal{P}_2$.
\item A maximal path of length one whose both ends are $k$-vertices, and this type will be denoted by $\mathcal{P}_3$. 
\item A maximal path of length two whose both ends are $(k-1)$-vertices, and this type will be denoted by $\mathcal{P}_4$. 
\end{enumerate}

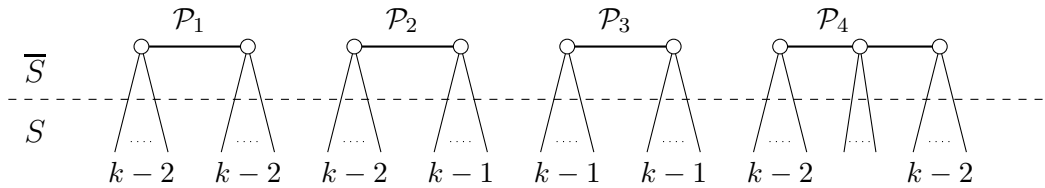
\begin{figure}[h]
\begin{center}
\begin{tikzpicture}[scale=.7]
\node at (-1,2.6) {$\overline{S}$};
\draw[dashed] (-1.5,2)--(18,2);
\node at (-1,1.4) {$S$};

\node at (1,3) (b) [circle,draw=black,fill=none,scale=0.5]{};\node at (1.9,3.5) {\small $\mathcal{P}_1$}; 
\node at (1,.6) {\small $k-2$}; \node at (3,.6) {\small $k-2$};
\node at (3,3) (c) [circle,draw=black,fill=none,scale=0.5]{};
\draw (0.5,1)-- (b) -- (1.5,1);  \draw[thick] (b) -- (c); \draw (3.5,1) -- (c) -- (2.5,1); 
\draw[dotted] (0.8,1.2) -- (1.2,1.2);\draw[dotted] (2.8,1.2) -- (3.2,1.2);

\node at (5,3) (b) [circle,draw=black,fill=none,scale=0.5]{};\node at (5.9,3.5) {\small $\mathcal{P}_2$}; 
\node at (5,.6) {\small $k-2$}; \node at (7,.6) {\small $k-1$};
\node at (7,3) (c) [circle,draw=black,fill=none,scale=0.5]{};
\draw (4.5,1)-- (b) -- (5.5,1);  \draw[thick] (b) -- (c); \draw (7.5,1) -- (c) -- (6.5,1); 
\draw[dotted] (4.8,1.2) -- (5.2,1.2);\draw[dotted] (6.8,1.2) -- (7.2,1.2);

\node at (9,3) (b) [circle,draw=black,fill=none,scale=0.5]{};\node at (9.9,3.5) {\small $\mathcal{P}_3$}; 
\node at (9,.6) {\small $k-1$}; \node at (11,.6) {\small $k-1$};
\node at (11,3) (c) [circle,draw=black,fill=none,scale=0.5]{};
\draw (8.5,1)-- (b) -- (9.5,1);  \draw[thick] (b) -- (c); \draw (11.5,1) -- (c) -- (10.5,1); 
\draw[dotted] (8.8,1.2) -- (9.2,1.2);\draw[dotted] (10.8,1.2) -- (11.2,1.2);
 
\node at (13,3) (h) [circle,draw=black,fill=none,scale=0.5]{};\node at (14,3.5) {\small $\mathcal{P}_4$};
\node at (13,.6) {\small $k-2$}; \node at (16,.6) {\small $k-2$}; 
\node at (14.5,3) (i) [circle,draw=black,fill=none,scale=0.5]{};
\node at (16,3) (j) [circle,draw=black,fill=none,scale=0.5]{};
\draw (12.5,1)-- (h) -- (13.5,1); \draw (14.2,1) -- (i) -- (14.8,1); \draw[thick] (h) -- (i) -- (j); \draw (16.5,1) -- (j) -- (15.5,1); 
\draw[dotted] (12.8,1.2) -- (13.2,1.2); \draw[dotted] (14.3,1.2) -- (14.7,1.2); \draw[dotted] (15.8,1.2) -- (16.2,1.2);
\end{tikzpicture}
\end{center}
\caption{The different types of maximal paths in $G[\overline{S}]$.}
 \label{fig:types}   
\end{figure}

Note that the interior vertex of a maximal path of type $\mathcal{P}_4$ is a $k$-vertex. We define the set $E$ to be a subset of $\overline{S}$ such that $E$ contains:\begin{itemize}
\item one and only one end vertex of each maximal path of type $\mathcal{P}_1$ and type $\mathcal{P}_3$,
\item the end vertex, which is a $k$-vertex, of each maximal path of type $\mathcal{P}_2$,
\item every end vertex of a path of type $\mathcal{P}_4$, and 
\item every isolated vertex of $G[\overline{S}]$.
\end{itemize}

Clearly, by the maximality of the paths, $E$ is an independent set of $G$. Moreover, $E'= \overline{S}\setminus E$ is also an independent set of $G$. We have the following remark:
\begin{remark}\label{rk3.2}
\begin{enumerate}
    \item  Let $x$ be an end vertex of a maximal path of type $\mathcal{P}_3$. Then $x$ is a $k$-vertex and it has a unique neighbor in $S_i$ for every $i$, $1\leq i\leq k-2$, except one. Set $l_x\in \{1,\dots, k-2\}$ such that $x$ has two neighbors in $S_{l_x}$. By Claim \ref{c3.1.1}, each unique neighbor of $x$ in $S_i$, $1\leq i\neq l_x\leq k-2$, is a $k$-vertex. Hence $x$ is a $k$-vertex having $(k-2)$ $k$-neighbors in $G$ and the two neighbors of $x$ in $S_{l_x}$ are both non $k$-vertices. 
     \item If $x$ is the interior vertex of a maximal path of type $\mathcal{P}_4$, then $x$ has two neighbors in $\overline{S}$ and so $x$ has a unique neighbor in $S_i$ for every $i$, $1\leq i\leq k-2$. Hence by Claim \ref{c3.1.1}, each neighbor of $x$ in $S_i$ is  a $k$-vertex for every $i$, $1\leq i\leq k-2$. \end{enumerate}
\end{remark}

 Color each vertex in $S_i$ by $1_i$ for every $i$, $1\leq i\leq k-2$, and color each vertex in $E$ by $1_{k-1}$.  We still need to color the vertices of $E'$, and this step will be possible thanks to the following claim:
 \begin{claim}
     The distance between any two vertices in $E'$ is at least three.
 \end{claim}
\begin{proof}
    Let $x$ and $y$ be two vertices in $E'$. Since $x$ and $y$ are not on the same maximal path, then $x$ and $y$ are not adjacent. Thus we only need to prove that $x$ and $y$ have no common neighbor. Suppose to the contrary that $x$ and $y$ have  a common neighbor, say $u$. Clearly, since $x$ and $y$ are not on the same maximal path in $G[\overline{S}]$, then $u$ is  in $S$. Hence there exists $m$, $1\leq m\leq k-2$, such that $u\in S_m$. We will consider the independent set $S'_m=(S_m\setminus (N(x)\cup N(y))\cup \{x,y\}$ in $G[V(G)\setminus (S_1\cup \dots \cup S_{m-1})]$ during our treatment of each of the following cases:
    \begin{enumerate}
        \item $x$ and $y$ are both $(k-1)$-vertices.\\
         As $x$ (resp. $y$) is a $(k-1)$-vertex, then $x$ (resp. $y$) is an end vertex of a path of type $\mathcal{P}_1$ or $\mathcal{P}_2$. Consequently, $x$ (resp. $y$) has a unique neighbor in $S_i$ for every $i$, $1\leq i\leq k-2$. Thus $(N(x)\cup N(y))\cap S_m=\{u\}$,  and so $\phi_m(S'_m)\geq \phi_m(S_m)-1+1.2$, a contradiction. 
        \item $x$ is a $(k-1)$-vertex and $y$ is a $k$-vertex.\\
        As above, $u$ is the unique neighbor of $x$ in $S_m$. We will distinguish two cases according to the number of neighbors of $y$ in $\overline{S}$:
        \begin{enumerate}
            \item $y$ is an end vertex of a path of type $\mathcal{P}_3$. \\
               If $m\neq l_y$, then $u$ is the unique neighbor of $y$ in $S_m$, and so $u$ is a $k$-vertex by Claim \ref{c3.1.1}. Thus  $\phi_m(S'_m)=\phi_m(S_m)-1+1.6$, a contradiction. If $m=l_y$, then $u$ is one of the two neighbors of $y$ in $S_m$ which are both non $k$-vertices by Remark~\ref{rk3.2} (1). Therefore,  $\phi_m(S'_m)=\phi_m(S_m)-1.2+1.6$, a contradiction.
            \item $y$ is the interior vertex of a path of type $\mathcal{P}_4$.\\
             In this case, $N(y)\cap S_m=\{u\}$ since $y$ has a unique neighbor in $S_i$ for every $i$, $1\leq i\leq k-2$. By Remark~\ref{rk3.2} (2), the neighbor of $y$ in $S_i$ is a $k$-vertex for every $i$, $1\leq i\leq k-2$, and so $\phi_m(S'_m)=\phi_m(S_m)-1+1.6$, a contradiction.
            \end{enumerate}
        \item $x$ and $y$ are both $k$-vertices.\\
         We need here to consider three cases:
        \begin{enumerate}
            \item $x$ and $y$ are both end vertices of a maximal path of type $\mathcal{P}_3$.\\
              If $m\notin \{l_x, l_y\}$, then $u$ is the unique neighbor of $x$ and $y$ in $S_m$. By Claim \ref{c3.1.1}, $u$ is a $k$-vertex and so $\phi_m(S'_m)=\phi_m(S_m)-1+2$, a contradiction. If $m=l_x$ and $l_x\neq l_y$, then $u$ is the unique neighbor of $y$ in $S_m$, and so $u$ is a $k$-vertex. But this contradicts the fact that both neighbors of $x$  in $S_{l_x}$ are non $k$-vertices. In the same way, we can reach a contradiction for the case $m=l_y$ and $l_x\neq l_y$. Finally, for the case $m=l_x=l_y$, $x$ (resp. $y$) has two neighbors in $S_m$ that are both non $k$-vertices by Remark~\ref{rk3.2} (1). In this case, $\phi'(S'_m)\geq \phi(S_m)-3*0.6+2$, a contradiction. 
            \item $x$ is an end vertex of a path of type $\mathcal{P}_3$ and $y$ is the interior vertex of a path of type $\mathcal{P}_4$.\\
             By Remark~\ref{rk3.2} (2), $u$ is the unique neighbor of $y$ in $S_m$ and $u$ is a $k$-vertex. Thus $m\neq l_x$ by Remark~\ref{rk3.2} (1). Consequently, $u$ is the unique neighbor of $x$ in $S_m$. Hence  $\phi'(S'_m)=\phi(S_m)-1+2$, a contradiction.
            \item $x$ and $y$ are both interior vertices of a maximal path of type $\mathcal{P}_4$. \\
             Hence $u$ is the unique  neighbor of $x$ (resp. $y$) in $S_m$ and $u$ is a $k$-vertex by Claim \ref{c3.1.1}. Thus  $\phi'(S'_m)=\phi(S_m)-1+2$, a contradiction.
        \end{enumerate}
        \end{enumerate}
        \end{proof}
     
     Finally, color the vertices of $E'$ by 2 and so we obtain a $(1^{k-1},2)$-packing coloring of $G$.    
\end{proof}

We will see in Section~\ref{s:concl} that the result of Theorem~\ref{tsat} is tight.
    
\section{$(k-1)$-Saturated $k$-Degree Graphs}  

We now turn our attention to the case $t=k-1$; i.e., to $(k-1)$-saturated $k$-degree graphs. We are going to use a technique similar to the one of the proof of Theorem~\ref{tsat}, but with a different weight function on the independent sets. We are then obliged to use $k-2$ more colors 2 than for the case $t\le k-2$.

\begin{theorem}\label{k-1sat}
    Let $G$ be a $(k-1)$-saturated $k$-degree graph, $k\geq 4$,  then $G$ is $(1^{k-1},2^{k-1})$-packing colorable. 
\end{theorem}
\begin{proof}
Without loss of generality, suppose $G$ is a connected $(k-1)$-saturated $k$-degree graph.
Let $T$ be an independent set in $G$. We denote by   $X_1(T)$  the set of $k$-vertices  in $T$ and by $Y_1(T)$  the set of non $k$-vertices in $T$. We define $\phi_1(T)=|X_1(T)|+0.4|Y_1(T)|$.
 
 Let $S_1$ be an independent set such that $\phi_1(S_1)\geq \phi_1(T)$ for every independent set $T$ in $G$. 
 Clearly, by the maximality of $\phi_1(S_1)$, each vertex in $V(G)\setminus S_1$ has a neighbor in $S_1$. 
   We are going to introduce  a sequence of independent sets $S_2,\dots,S_{k-2}$ in $G\setminus S_1$ such that $S_i$ is defined starting from $i=2$ as follows: Let $T$ be an independent set in $G[V(G)\setminus(S_1\cup\dots\cup S_{i-1}) ]$. We denote by   $X_i(T)$  the set of $k$-vertices  in $T$ and by $Y_i(T)$  the set of non $k$-vertices in $T$. We define $\phi_i(T)=|X_i(T)|+0.4|Y_i(T)|$. Let $S_i$ be an independent set in $G[V(G)\setminus(S_1\cup\dots \cup S_{i-1}) ]$ such that 
 $\phi_i(S_i)\geq \phi_i(T)$ for every independent set $T$ in $G[V(G)\setminus(S_1\cup\dots \cup S_{i-1}) ]$. \\
 Clearly, by the maximality of $\phi_i(S_i)$, $1\leq i \leq k-2$, each vertex in $G[V(G)\setminus(S_1\cup\dots \cup  S_{i}) ]$ has a neighbor in $S_j$ for every $j$, $1\leq j< i$. Let $S=S_1\cup\dots \cup S_{k-2}$ and let $\overline{S}=V(G)\setminus (S_1\cup \dots S_{k-2})$. Let $u$ be a vertex in $\overline{S}$ and let $v$ be a vertex in  $S_i$ for some $i$, $1\leq i\leq k-2$. We say $v$ is an {\em $i$-father} of $u$ if   $v$ and $u$ are adjacent. In general, if $u$ is a vertex in $\overline{S}$ and $v$ is a vertex in $S$ such that $u$ and $v$ are adjacent, then we say $v$ is a father of $u$.  We have the following result:
\begin{claim}\label{cl3.2.1}
 If $u$ is a $k$-vertex in $\overline{S}$ such that $u$ has a unique $i$-father, say $v$,  for some $i$, $1\leq i\leq k-2$, then $v$ is a $k$-vertex.
 \end{claim}
 \begin{proof}
     Suppose to the contrary that $v$ is a non $k$-vertex, then  $S'_i=(S_i\setminus\{v\})\cup\{u\}$ is an independent set in $G\setminus (S_1\cup \dots... \cup S_{i-1})$ with $\phi_i(S'_i)= \phi_i(S_i)-0.4+1$, a contradiction.  
 \end{proof}
\begin{remark}~\label{rkn}
Clearly, each $m$-vertex in $\overline{S}$, $m\leq k-2$, has no neighbor in $\overline{S}$, each $(k-1)$-vertex has at most one neighbor in $\overline{S}$ and each $k$-vertex has at most two neighbors in $\overline{S}$. Moreover, at most one neighbor of a $k$-vertex in $\overline{S}$ is a $k$-vertex. In fact, suppose that $x$ is a $k$-vertex in $\overline{S}$ such that two of the neighbors of $x$ in $\overline{S}$ are $k$-vertices. In this case, $x$ has a unique $i$-father  for every $i$, $1\leq i\leq k-2$.  Then since $G$ is $(k-1)$-saturated, there exists $i$, $1\leq i\leq k-2$, such that the $i$-father of $x$ in $S_i$ is a non $k$-vertex, a contradiction by Claim~\ref{cl3.2.1} \end{remark}
 Thus we can deduce that $G[\overline{S}]$ contains no cycle and each maximal path in $G[\overline{S}]$ is a path of length at most three.
 Note that, by Remark~\ref{rkn}, each interior vertex of a maximal path of length 2 or 3 is a $k$-vertex.
Consequently, each end vertex of a path of length 3 is a $(k-1)$-vertex. We define the set $E$ to be a subset of $\overline{S}$ such that $E$ contains:
 \begin{itemize}
 \item each isolated vertex of $G[\overline{S}]$;
\item exactly one vertex of each maximal path of length one, the one of degree $k$ if exists;
\item the two end vertices of each maximal path of length 2; 
\item exactly one end vertex and one interior vertex of each path of length 3 such that the two chosen vertices are not adjacent. 
\end{itemize}
Clearly, by the maximality of the paths and by the definition of the vertices in $E$, $E$ is an independent set. For the same reason, $\overline{S}\setminus E$ is an independent set.
 Color each vertex in $S_i$ by $1_i$, $1\leq i\leq k-2$, and color each vertex in $E$ by $1_{k-1}$.  
 For each vertex $x$ in $G$,  we will denote by $c(x)$ the color of $x$ if $x$ is colored and we define $C_1(x)=\{c(v): v\; \text{is a colored neighbor of }\; x\}$. We still need to color the vertices of $\overline{S}\setminus E$, but first we have the following remark:
 \begin{remark}\label{rk3.4} Let $x$ be a vertex in $\overline{S}\setminus E$, then: \begin{enumerate}
     \item If $x$ is a  $(k-1)$-vertex, then $x$ has a unique neighbor in $\overline{S}$, and so $x$ has a unique $i$-father for each $i$, $1\leq i\leq k-2$.
     \item If $x$ is a  $k$-vertex, then  $x$ is either an end vertex  or an interior vertex of a maximal path in $G[\overline{S}]$.  
     If $x$ is an end vertex, then $x$ has a unique neighbor in $\overline{S}$, which is a $k$-vertex. Thus $x$ has a unique $i$-father for every $i$, $1\leq i\leq k-2$, but except one. Let $l_x$, $1\leq l_x \leq k-2$, be such that $x$ has two $l_x$-fathers. At most one of the two $l_x$-fathers is a $k$-vertex. In fact, $x$ has a unique $i$-father, $1\leq i\neq l_x\leq k-2$, and so, by Claim~\ref{cl3.2.1}, each $i$-father of $x$, $i\neq l_x$, is a $k$-vertex.  Nevertheless, one of the $l_x$-fathers of $x$ is a $k$-vertex. Actually, if each $l_x$-father of $x$ is a non $k$-vertex, then $S'_{l_x}=(S_{l_x}\setminus N(x))\cup\{x\}$ is an independent set in $G\setminus (S_1\cup \dots... \cup S_{l_x-1})$ with $\phi_{l_x}(S'_{l_x})= \phi_{l_x}(S_{l_x})-2*0.4+1$, a contradiction. 
     
     If  $x$ is an interior vertex of a maximal path in $G[\overline{S}]$, then $x$ has two neighbors in $\overline{S}$. Accordingly, $x$ has a unique $i$-father for every $i$, $1\leq i\leq k-2$. Thus an $i$-father of $x$ is a $k$-vertex for every $i$, $1\leq i\leq k-2$. 
 \end{enumerate}
 \end{remark}
 
  Two vertices $x$ and $y$ in $\overline{S}\setminus E$ are said to be {\em siblings} if $x$ and $y$ have a common father in $S$.  In this case, we say $x$ is a sibling of $y$. More precisely, for $i\in \{1,\dots, k-2\}$, we say $x$ is an $i$-sibling of $y$ if $x$ and $y$ have a common $i$-father. An  $i$-father $u$ of a vertex in $\overline{S}\setminus E$ is said to be a {\em good} $i$-father if $\{1_1,\dots, 1_{k-1}\}\setminus (C_1(u)\cup \{1_i\})\neq \emptyset$. Otherwise, we say $u$ is a {\em bad} $i$-father. \\
 
  Let $x\in \overline{S}\setminus E$. If there exists $i$, $1\leq i\leq k-2$, such that each $i$-father of $x$ is a good $i$-father, then recolor each $i$-father $u$ of $x$, by $1_j$  and then color $x$ by $1_i$, where $1_j\in \{1_1,\dots, 1_{k-1}\}\setminus (C_1(u)\cup \{1_i\})$. If all the vertices in $\overline{S}\setminus E$ are colored this way, then we are done. Else, let $E'$ be the set of uncolored vertices in $\overline{S}\setminus E$.  \\
 
  Let $x$ be a vertex in $E'$, then $x$ has at least one bad $i$-father for every $i$, $1\leq i\leq k-2$. If $x$ is a $(k-1)$-vertex, then, by Remark~\ref{rk3.4}, each $i$-father of $x$ is a bad father for every $i$, $1\leq i\leq k-2$. On the other hand, if $x$ is a $k$-vertex, then, by Remark~\ref{rk3.4}, each father of $x$ is a bad father except at most one. Note that if $u$ is a bad $i$-father of $x$ for some $i$, $1\leq i\leq k-2$, then  $C_1(u)=\{1_1,\dots, 1_{k-1}\}\setminus\{1_i\}$. Consequently, $u$ has at most two neighbors in $E'$.\\
 
   In order to color the vertices of $E'$, we are going to prove that $\Delta (G^2[E'])\leq k-1$ and $G^2[E']$ contains no complete subgraph of order $k$, where $G^2[E']$ is the subgraph of the square graph of $G$ induced by $E'$. In fact, this means that $G^2[E']$ is $(k-1)$-colorable by Brooks' Theorem and so the vertices of $E'$ can be colored by $2_1,\dots, 2_{k-1}$. To reach the desired result, we need the following sequence of Claims:

 \begin{claim}\label{cm3.2.2}
     Let $x$ and $y$ be two  vertices in $E'$ such that $x$ is an $i$-sibling of $y$ for some $i$, $1\leq i\leq k-2$, then:\begin{enumerate}
         \item If $x$ is a  $(k-1)$-vertex, then $y$ is the unique  $i$-sibling of $x$ in $E'$.
         \item If $x$ is a $k$-vertex and $y$ is a  $(k-1)$-vertex, then $x$  has two $i$-fathers. 
         \item If $x$ and $y$ are both  $k$-vertices, then $x$ (resp. $y$) has two $i$-fathers. Moreover, the common $i$-father of $x$ and $y$ is not a $k$-vertex and not a bad one
     \end{enumerate}
 \end{claim}
\begin{proof}
\begin{enumerate}
    \item Let $u$ be the common $i$-father of $x$ and $y$. Since $u$ is the unique $i$-father of $x$, then $u$ is a bad father, and so $(\{1_1,\dots,1_{k-1}\}\setminus \{1_i\})\subseteq C_1(u)$. Consequently, $u$ has no  neighbor in $E'$ other than $x$ and $y$. Hence $y$ is the unique $i$-sibling of $x$ in $E'$.
   
    \item Suppose to the contrary that $x$ has a unique $i$-father, say $u$. Then $S'_i=(S_i\setminus \{u\})\cup \{x,y\})$ is an independent set in $G\setminus \{S_1,\dots, S_{i-1}\}$ with $\phi_i(S'_i)\geq \phi(S_i)-1+1.4$, a contradiction.

    \item Suppose to the contrary that $x$ has a unique $i$-father, then $|(N(x)\cup N(y))\cap S_i|\leq 2$. Recall that, by Remark~\ref{rk3.4} (2),  only one of the $i$-fathers of $y$ is a $k$-vertex.  Then  $S'_i=(S_i\setminus ( N(y))\cup \{x,y\}$ is an independent set in  $G\setminus (S_1\cup \dots... \cup S_{i-1})$ with $\phi_i(S'_i)=\phi_i(S_i)-1.4+2$, a contradiction. In the same way, we can prove that $y$ has two $i$-fathers.\\
     Suppose to the contrary that the common $i$-father of $x$ and $y$ is a $k$-vertex, then the other $i$-father of $x$ (resp. $y$) is  not a $k$-vertex by Remark~\ref{rk3.4} (2). Hence $S'_i=(S_i\setminus (N(x)\cup N(y))\cup \{x,y\}$ is an independent set in $G\setminus \{S_1,\dots, S_{i-1}\}$ with $\phi_i(S'_i)=\phi_i(S_i)-1.8+2$, a contradiction. 

Let $u$ be the common $i$-father of $x$ and $y$ and suppose to the contrary that $u$ is a bad $i$-father. Since $u$ is a bad father, then $(\{1_1,\dots,1_{k-1}\}\setminus \{1_i\})\subseteq C_1(u)$. Consequently, $u$ is a $k$-vertex, a contradiction. \end{enumerate}
\end{proof}

\begin{remark}\label{rk3.5} Let $x$ be a vertex in $E'$. We have the following observations:  \begin{enumerate}
\item If $x$ is a $k$-vertex such that $x$ is an interior vertex of a maximal path in $G[\overline{S}]$, then $x$ has no sibling in $E'$. In fact, by Remark~\ref{rk3.4} (2), $x$ has a unique $i$-father for every $i$, $1\leq i\leq k-2$. Thus by  Claim~\ref{cm3.2.2} (2,3), $x$ has no sibling in $E'$. Consequently $x$ is at distance at least three from each vertex in $E'$ except at most one. This exception is due to the fact that $x$ may be an interior vertex of a maximal path of length three and so one of the end vertices of this maximal path  may belong to $E'$. Consequently, $x$ has at most one neighbor in $G^2[E']$.
\item If $x$ is a  $(k-1)$-vertex, then  by Claim~\ref{cm3.2.2} (1), $x$ has at most $(k-2)$  siblings in $E'$ since $x$ has at most one $i$-sibling in $E'$ for every $i$, $1\leq i\leq k-2$. In addition, if $x$ is an end vertex of a path of length three and an interior vertex of this path belongs to $E'$, then $x$ and this interior vertex are neighbors in $G^2[E']$. Consequently, each  $(k-1)$-vertex in $E'$ has at most $(k-1)$ neighbors in $G^2[E']$. Besides, if $x$ is not an end vertex of a maximal path of length three then $x$ has at most $(k-2)$ neighbors in $G^2[E']$. 
\item  If $x$ is a  $k$-vertex that has an $i$-sibling in $E'$ for some $i$, $1\leq i\leq k-2$, then $i=l_x$. Moreover, $x$ has no $j$-sibling for every $j\neq i$, $1\leq j\leq k-2$, since $x$ has a unique $j$-father. 
\end{enumerate}\end{remark}

In addition, we have the following important result:
\begin{claim}
Each $k$-vertex in $E'$ has at most $(k-1)$ siblings in $E'$.
\end{claim}
\begin{proof}
Let $x$ be a $k$-vertex in $E'$ and suppose to the contrary that $x$ has $k$  siblings in $E'$. Then by Remark~\ref{rk3.5} (3), each sibling of $x$ in $E'$ is an  $l_x$-sibling of $x$.  Consequently, by Claim~\ref{cm3.2.2} (1)  and since $k\geq 4$, $x$ has at most one sibling which is a $(k-1)$-vertex. In fact, if $x$ has two siblings that are $(k-1)$-vertices, say $y$ and $z$, then, by Claim~\ref{cm3.2.2} (1), $x$ is the unique sibling of $y$ (resp. $z$), and so the common $l_x$-father of $x$ and $y$ is distinct from that of $x$ and $z$. But then we get that both $l_x$-fathers of $x$ are $k$-vertices since each father of a $(k-1)$-vertex in $E'$ is bad, a contradiction.  Let $x_1,\dots, x_k$ be the  $l_x$-siblings of $x$ in $E'$ and let $u$ and $v$ be the $l_x$-fathers of $x$. Without loss of generality, suppose that $u$ is a bad $l_x$-father of $x$ and $x_2,\dots, x_k$ are all $k$-vertices. By Claim~\ref{cm3.2.2} (3), $u$ is not an $l_x$-father of $x_i$ for every $i$, $2\leq i\leq k$. Thus  $v$ is the $l_x$-father of $x$ and $x_i$ for every $i$, $2\leq i\leq k$. Hence $v$ is a $k$-vertex whose neighbors are all $k$-vertices, a contradiction.     
\end{proof}

 Thus $\Delta(G^2[E'])\leq k-1$. We are in front of the final step:
 
\begin{claim}
  $G^2[E']$ has no complete subgraph on $k$ vertices.   
\end{claim}
\begin{proof}
 Suppose that $G^2[E']$ has a complete subgraph $K$ on $k$ vertices.  We will prove first that each vertex of $K$ is a $k$-vertex. Suppose to the contrary that $K$ has a $(k-1)$-vertex, say $x$. Since $x$ has $k-1$ neighbors in $G^2[E']$, then by Remark~\ref{rk3.5} (2), one of the neighbors of $x$ in $G^2[E']$, say $y$, is an interior vertex of the path of length three to which $x$ belongs. Thus $y$ is a vertex of $K$ and so $y$ has $(k-1)$ neighbors in $G^2[E']$, a contradiction by Remark~\ref{rk3.5} (1). 
 
 Let $x_1,\dots,x_k$ be the vertices of $K$. By Remark~\ref{rk3.5} (3),  $x_2,\dots, x_{k}$ are $l_{x_1}$-siblings of $x_1$. Let $u$ and $v$ be the $l_x$-fathers of $x_1$ and suppose that $v$ is a bad father. Then $u$ is the common $l_{x_1}$-father of $x_1$ and $x_i$ for every $i$, $2\leq i\leq k$. Hence, $u$ is a $k$-vertex whose neighbors are all $k$-vertices, which is a contradiction.
\end{proof}

 Thus by Brooks' Theorem,  $G^2[E']$ is $(k-1)$-colorable. Color the vertices of $E'$ by the colors $2_1, \dots, 2_{k-1}$ to obtain a $(1^{k-1},2^{k-1})$-packing coloring of $G$. 
\end{proof}

Note that we proved the above result for $k\ge 4$ only since for $k=3$, we would have to show that $G^2[E']$ does not contain an odd cycle. Moreover, a better result is already known in this case: Every 2-saturated subcubic graph is $(1,1,2,3)$-packing colorable~\cite{MT2}.


\section{ $k$-Degree Graphs}

We finish with the general case of $k$-degree graphs and prove the following result (using one more color 2 than for the $(k-1)$-saturated case).
\begin{theorem}\label{kdeg}
    Every $k$-degree graph, $k\geq 3$, is $(1^{k-1},2^k)$-packing colorable.
\end{theorem}
\begin{proof}
Without loss of generality, suppose $G$ is connected. Let $T_1,\dots,T_{k-1}$ be pair-wisely disjoint  independent sets in $G$ such that $T_i\neq \emptyset$ for every $i$, $1\leq i\leq k-1$. We call $T=T_1\cup\dots \cup T_{k-1}$ a $(k-1)$-independent set of $G$. Let $\overline{T}=V(G)\setminus T$. We define $\Theta(T)=\{xy\in E(G):\; x\in T \; \text{and} \; y\in \overline{T}\}$ and   $\theta(T)=|\Theta(T)|$. For a vertex $x$ in $\overline{T}$, we define $\Theta_T(x)=\{ux\in E(G):\; u\in T\}$ and $\theta_T(x)=|\Theta_T(x)|$. Clearly, $\theta(T)=\sum_{x\in \overline{T}}\theta_T(x). $

Let $S_1,\dots, S_{k-1}$ be  pair-wisely disjoint  independent sets such that $S_i\neq \emptyset$ for every $i$, $1\leq i\leq k-1$. and let $S=S_1\cup\dots \cup S_{k-1}$. Suppose that $S$ is chosen such that $|S|$ is maximum among all $(k-1)$-independent sets. Subject to this condition, suppose that $S$ is chosen such that $\theta(S)$ is minimum.

Our plan is to color each vertex in $S_i$ by $1_i$, $1\leq i\leq k-1$, and then we will use the colors $2_1,\dots, 2_{k}$ for the coloring of the vertices of $\overline{S}$.

Let $x$ be a vertex of $\overline{S}$ and let $y$ be a vertex in $S_i$ for some $i$, $1\leq i\leq k-1$. We say $y$ is an {\em $i$-father} of $x$ if $x$ and $y$ are adjacent. For brevity, we say $y$ is a {\em father} of $x$. Note that $x$ has at least one $i$-father for every $i$, $1\leq i\leq k-1$. In fact, if there exists $i$, $1\leq i\leq k-1$, such that $x$ has no $i$-father, then  $S'=S'_1\cup \dots \cup S'_{k-1}$, where $S'_j=S_j$ for every $j$, $1\leq j\neq i\leq k-1$, and $S'_i=S_i\cup\{x\}$, is a $(k-1)$-independent set  with $|S|<|S'|$, a contradiction. Consequently every vertex $x$ in $\overline{S}$ is either a $(k-1)$-vertex or a $k$-vertex.  Besides, if $x$ is a $(k-1)$-vertex, then $x$ has no neighbor in $\overline{S}$ and $x$ has a unique $i$-father for every $i$, $1\leq i\leq k-1$.  On the other hand, if $x$ is a $k$-vertex, then $x$ has at most one neighbor in $\overline{S}$. A $k$-vertex $x$ in $\overline{S}$ is said to be a {\em heavy $k$-vertex} if $x$ has a neighbor in $\overline{S}$. The neighbor in $\overline{S}$ of a heavy $k$-vertex, which is also a heavy $k$-vertex, will be called the {\em twin} of $x$. If $x$ is a heavy $k$-vertex, then $x$ has a unique $i$-father for every $i$, $1\leq i\leq k-1$. Moreover, if $x$ is a non-heavy $k$-vertex then there exists $i$, $1\leq i\leq k-1$, such that $x$ has two $i$-fathers. Such an $i$ will be denoted by $m_S(x)$. Clearly, $x$ has a unique $j$-father for every $j\neq m_S(x)$. 

Let $x$ be a vertex of $\overline{S}$ and let $u$ be an $i$-father of $x$ for some $i$, $1\leq i\leq k-1$. We say $u$ is a {\em bad $i$-father} of $x$ with respect to $S$ if $N(u)\cap S_j\neq\emptyset$ for every $j$, $1\leq j\neq i\leq k-1$.  For brevity, we say $u$ is a bad father of $x$ unless when needed. Thus if $u$ is a bad $i$-father of $x$ then $u$ has at most two neighbors in $\overline{S}$; that is $u$ is an $i$-father of at most one vertex in $\overline{S}\setminus\{x\}$.    We have the following result:

\begin{claim}\label{c0}
Let $x$ be a vertex in $\overline{S}$, then:
\begin{enumerate}
    \item every father of $x$ is bad if $x$ is a $(k-1)$-vertex or a heavy $k$-vertex, and
    \item  every $i$-father of $x$ is bad for every $i$, $1\leq i\neq m_S(x)\leq k-1$, and at most one of the $m_S(x)$-fathers is not bad if $x$ is a non-heavy $k$-vertex.
\end{enumerate}
\end{claim}
\begin{proof}
\begin{enumerate}
    \item Let $u$ be an $i$-father of $x$ for some $i$, $1\leq i\leq k-1$, then $u$ is the unique $i$-father of $x$. Suppose to the contrary that $u$ is not a bad father, then there exists $t$, $1\leq t\neq i\leq k-1$, such that $u$ has no neighbor in $S_t$. Therefore $S'=S'_1\cup \dots \cup S'_{k-1}$, where $S'_j=S_j$ for every $j$, $j\in \{1,\dots, k-1\}\setminus \{i,t\}$, $S'_i=(S_i\setminus\{u\})\cup \{x\}$ and $S'_t=S_t\cup\{u\}$, is a $(k-1)$-independent set  with $|S|<|S'|$, a contradiction. 
    \item Let $u$ be an $i$-father of $x$ for some $i$, $1\leq i\leq k-1$. If $i\neq m_S(x)$, then $u$ is the unique $i$-father of $x$, and, as above in (1), we can prove that $u$ is a bad $i$-father. Else, let $v$ be the $m_S(x)$-father of $x$ with $v\neq u$. If both $u$ and $v$ are not bad fathers, then let $t$ (resp. $t'$), $\{t,t'\}\subset\{1,\dots, k-1\}\setminus\{m_S(x)\}$, such that $u$ (resp. $v$) has no neighbor in $S_t$ (resp. $S_{t'}$). Consequently,   $S'=S'_1\cup \dots \cup S'_{k-1}$, where $S'_j=S_j$ for every $j$, $j\in \{1,\dots, k-1\}\setminus \{i,t,t'\}$, $S'_i=(S_i\setminus\{u,v\})\cup \{x\}$, $S'_t=S_t\cup\{u\}$, and $S'_{t'}=S_{t'}\cup\{v\}$, is a $(k-1)$-independent set  with $|S|<|S'|$, a contradiction. 
 \end{enumerate}   
\end{proof}

Color the vertices of $S_i$ by $1_i$ for every $i$, $1\leq i\leq k-1$. 
We will prove now that the vertices of $\overline{S}$ can be colored by $2_1,\dots, 2_{k}$. In order to reach this result, we will prove first $\Delta (G^2[\overline{S}])\leq k$  and then $G^2[\overline{S}]$ contains no complete subgraph of order $k+1$. This result allows us to prove, using Brooks' Theorem, that $G^2[\overline{S}]$ is $k$-colorable which means that $G[\overline{S}]$ can be colored by the colors $2_1, \dots, 2_{k}$. 

Let $x$ and $y$ be two vertices of $\overline{S}$. We say $x$ is  a sibling of $y$  if one of the conditions is satisfied:
\begin{itemize}
    \item $y$ is a heavy $k$-vertex and $x$ is the twin of $y$.
    \item $x$ and $y$ have a common father.
\end{itemize}
Moreover, if $x$ is not the twin of $y$ but siblings and have a common $i$-father, then we say $y$ is an $i$-sibling of $x$. In order to prove  $\Delta (G^2[\overline{S}])\leq k$, it is important to count the number of siblings of each vertex in $\overline{S}$.

\begin{claim}\label{c1} Let $x$ be a vertex in $\overline{S}$. Then, for every $i$, $1\leq i\leq k-1$,  an $i$-father of $x$ is a father of at most one vertex in $\overline{S}\setminus\{x\}$  if $x$ is a $(k-1)$-vertex or $x$ is a heavy vertex. 
\end{claim}
\begin{proof}
By Claim~\ref{c0}, an $i$-father of $x$ is bad for every $i$, $1\leq i\leq k-1$. Since each bad father has $k-2$ neighbors in $S$, then the result follows for both cases.\end{proof}

Then $x$ has at most $k - 1$ siblings if $x$ is a $(k-1)$-vertex, and $x$ has at most $k - 1$ siblings other than its twin if $x$ is a heavy $k$-vertex. Consequently, $d_{G^2[\overline{S}]}(x)\leq k-1$ if $x$  is a $(k-1)$-vertex and $d_{G^2[\overline{S}]}(x)\leq k$ if $x$ is a heavy $k$-vertex. We still need to prove $d_{G^2[\overline{S}]}(x)\leq k$ if $x$ is a non-heavy $k$-vertex in $\overline{S}$. Here is a sequence of claims that will lead to the desired result:

\begin{claim} \label{c2}
    Let $x$ be a non-heavy $k$-vertex. Then $x$ has at most one $i$-sibling for every $i$, $1\leq i\neq m_S(x)\leq k-1$. Besides, let $u$ and $v$ be the $m_S(x)$-fathers of $x$ and suppose without loss of generality that $v$ is a bad one.  Then if $v$ is a father of a vertex in $\overline{S}\setminus\{x\}$, then $u$ is a father of at most one vertex in $\overline{S}\setminus\{x\}$. 
\end{claim}
\begin{proof}
   For $i\neq m_S(x)$,  every $i$-father of $x$ is a bad one, by Claim ~\ref{c0}, and so $x$ has hat at most one $i$-sibling.

We still need to prove that if $v$ is a father of a vertex in $\overline{S}\setminus\{x\}$, then $u$ is a father of at most one vertex in $\overline{S}\setminus\{x\}$. Suppose to the contrary that $u$ is a father of two vertices in $\overline{S}\setminus\{x\}$. Then there exists $t$, $1\leq t\neq m_S(x)\leq k-1$, such that $u$ has no neighbor in $S_t$. Consequently, $S'=S'_1\dots S'_{k-1}$, where $S'_i=S_i$ for every $i$, $i\in \{1,\dots, k-1\}\setminus\{m_S(x),t\}$, $S'_{m_S(x)}=(S_{m_S(x)}\setminus\{u,v\})\cup \{x\}$ and  $S'_t=S_t\cup \{u\}$, is a $(k-1)$-independent set with $|S'|=|S|$. Remark  that $v\in \overline{S'}$ while $\{x,u\}\subset S'$. We have $\Theta (S')=(\Theta (S)\setminus \Theta_S(x))\cup \Theta_{S'}(v)$. As $x$ is a non-heavy $k$-vertex in $\overline{S}$, then $\theta_S(x)=k$, while $v$ has a neighbor in $\overline{S'}$ since $v$ is a father of a vertex in $\overline{S}\setminus\{x\}$ and so $\theta_{S'}(v)=k-1$. Consequently, $\theta(S')=\theta(S')-1$, a contradiction.  
\end{proof}

Let $x$ be a non-heavy $k$-vertex and let $u$ and $v$ be the $m_S(x)$-fathers of $x$ with $v$ is a bad one. Clearly, when $u$ is not a bad father, then $u$ may have more than one neighbor in $\overline{S}\setminus\{x\}$. If $u$ is a father of more than one vertex in $\overline{S}\setminus \{x\}$, then we say $x$ is a {\em dense} non-heavy $k$-vertex. In this case, $v$ has no neighbor in $\overline{S}\setminus\{x\}$ by Claim ~\ref{c2}. We have the following remark concerning the non-heavy $k$-vertices:
\begin{remark}\label{r1}
Let $x$ be a non-heavy $k$-vertex and let $u$ and $v$ be the two $m_S(x)$-fathers of $x$ with $v$ is a bad father of $x$. Then $x$ satisfies one of the following:\begin{itemize}
    \item $x$ has no $m_S(x)$-sibling.
    \item $x$ has a unique $m_S(x)$-sibling, say $y$, such that $v$ is the father of $y$.
    \item $x$ has only two $m_S(x)$-siblings, say $y$ and $z$,  such that the common father of $x$ and $y$ is distinct from that of $x$ and $z$.
    \item $x$ has exactly $l$ $m_S(x)$-siblings for some $l$, $1\leq l\leq k-1$, such that $u$ is the common father of $x$ and each of these siblings.
\end{itemize} \end{remark}

 Moreover, we have the following result:
\begin{claim} \label{c3}
    Let $x$ be a dense non-heavy $k$-vertex. Then $x$ has at most $k-1$ siblings. 
\end{claim}
\begin{proof}
   Let $u$ and $v$ be the $m_S(x)$-fathers of $x$ and suppose that $v$ is a bad one. Let $l$ be the number of $m_S(x)$-siblings with $l>1$, then, as $x$ is dense and by Claim~\ref{c2}, $u$ is the father of these $l$ siblings. Since $u$ has $l+1$ neighbors in $\overline{S}$ and $\Delta(G)\leq k$, then there exists $i_1, \dots, i_{l-1}$, $\{i_1,\dots, i_{l-1}\}\subset (\{1,\dots,k-1\}\setminus\{m_S(x)\})$,  such that $u$ has no neighbor in $S_{i_j}$ for every $j$, $1\leq j\leq l-1$.
   \\ We will prove now that $x$ has no $i_j$-sibling for every $j$,  $1\leq j\leq l-1$. Suppose to the contrary that there exists $t$,  $1\leq t\leq l-1$, such that $x$ has an $i_t$-sibling, say $y$, and let $v'$ be the common $i_t$-father of $x$ and $y$. Recall that $v'$ is the unique $i_t$-father of $x$ since $i_t\neq m_S(x)$. Let $S'=S'_1\cup \dots \cup S'_{k-1}$ such that $S'_j=S_j$ for every $j$, $j\notin \{m_S(x),i_t\}$, $S'_{m_S(x)}=S_{m_S(x)}\setminus\{u\}$, and $S'_{i_t}=S_{i_t}\cup \{u\}$. Since $u$ has no neighbor in $S_{i_t}$, then $S'$ is a $(k-1)$-independent set with $|S'|=|S|$ and $\theta(S')=\theta(S)$. We have now $m_{S'}(x)=i_t$ and $u$ and $v'$ are the two $m_{S'}(x)$-fathers of $x$. Since $u$ has $l$ neighbors in $\overline{S'}$ and $l>1$, then  $v'$ is the unique bad $m_{S'}(x)$-father of $x$ with respect to $S'$. But $v'$ is a father of a vertex in $\overline{S'}\setminus\{x\}$, $y$, then we get a contradiction by Claim ~\ref{c2}. Thus, $x$ has no $i_j$-sibling for every $j$,  $1\leq j\leq l-1$.  On the other hand, an $i$-father of $x$ is a bad one and so it has at most one neighbor in $\overline{S}\setminus\{x\}$ for every $i\in \{1,\dots,k-1\}\setminus(\{i_1,\dots, i_{l-1}\}\cup\{m_S(x)\})$. Hence, $x$ has at most $(k-2-(l-1))+l$ siblings. 
\end{proof}

 Consequently, we have the following observation:
 \begin{remark}\label{r2}
     If $x$ is a non-heavy $k$-vertex then $x$ has at most $k$ siblings. In fact, if $x$ is not dense, then $x$ has at most two $m_S(x)$-siblings and at most one $i$-sibling for every $i$, $i\neq m_S(x)$. Besides, if $x$ has $k$ siblings then each father of $x$ has a unique neighbor in $\overline{S}\setminus \{x\}$ by Claim~\ref{c2}, Claim~\ref{c3} and Remark~\ref{r1}.  
 \end{remark}

From all the above results, one can deduce that $\Delta(G^2[\overline{S}])\leq k$. We still need to prove that $G^2[\overline{S}]$ contains no complete subgraph on $k+1$ vertices. 

 To characterize the vertices of a complete subgraph of  $G^2[\overline{S}]$  on $(k+1)$ vertices, we need the following result: 

\begin{claim}\label{c4}
 Let $x$ and $y$ be two $t$-siblings for some $t$, $1\leq t\leq k-1$, then either $x$ or $y$ has two $t$-fathers. 
\end{claim}
\begin{proof}
    Suppose to the contrary that both have a unique $t$-father, say $u$. Accordingly, $S'=S'_1\cup \dots \cup S'_{k-1}$, where $S'_i=S_i$ for every $i$, $i\neq t$, and $S'_t=(S_t\setminus\{u\})\cup \{x,y\}$, then $S'$ is a $(k-1)$-independent set with $|S'|>|S|$, a contradiction.
\end{proof}

Let $P$ be a path in $G[\overline{S}\cup S_t]$ for some $t$, $1\leq t\leq k-1$. We say that $P$ is an {\em alternating $\overline{S}S_t$-path} if any two consecutive vertices are neither both in $\overline{S}$ nor in $S_t$. 
We have the following result:

\begin{claim}\label{c5}
Let $P=x_1\dots x_r$, $r\geq 3$, be an alternating $\overline{S}S_t$-path of maximal length for some $t$, $1\leq t\leq k-1$,  then $|\overline{S}\cap\{x_1,x_r\}|\leq 1$.

\end{claim}
\begin{proof}
Suppose to the contrary that $|\overline{S}\cap\{x_1,x_r\}|=2$. Consequently, $|P\cap \overline{S}|=|P\cap S_t|+1$. Moreover, since $P$ is of maximal length, then $x_2$ (resp. $x_{r-1}$) is the unique $t$-father  of $x_1$ (resp. $x_r$). Besides, whenever $x_l\in P\cap \overline{S}$, $3\leq l\leq r-2$, $x_{l-1}$ and $x_{l+1}$ are the  $t$-fathers of $x_{l}$. This means that $(S_t\setminus\{x_{l-1},x_{l+1}\})\cup \{x_l\}$ is an independent set   whenever $x_l\in P\cap \overline{S}$, $3\leq l\leq r-2$. Thus, $S'=S'_1\dots S'_{k-1}$, where $S'_i=S_i$ for every $i$, $1\leq i\neq t\leq k-1$, and $S'_t=(S_t\setminus(P\cap S_t))\cup (P\cap \overline{S})$, is a $(k-1)$-independent set with $|S'|=|S|+|P\cap \overline{S}|-|P\cap S_t|>|S|$, a contradiction. 
\end{proof}

Finally, we reach the desired result:
\begin{claim}
    $G^2[\overline{S}]$ contains no complete subgraph on $k+1$ vertices.
\end{claim}
\begin{proof}
    Suppose to the contrary that $G^2[\overline{S}]$ contains a complete subgraph on $k+1$ vertices, say $K$. Since a vertex $x$ of $K$ has $k$ siblings, then, by Claim~\ref{c0}, $x$ is a $k$-vertex. Moreover, $x$ is either a heavy $k$-vertex or $x$ is a non-heavy $k$-vertex which is not dense by Claim~\ref{c3}. Besides, by Claim~\ref{c4}, $K$ contains at most two heavy vertices: a heavy vertex and its twin. This result is due to the fact that a heavy vertex has a unique $i$-father for every $i$, $1\leq i\leq k-1$.  Moreover, by Remark~\ref{r2}, each father of a non-heavy $k$-vertex $x$ in $K$  has a unique neighbor in $\overline{S}\setminus\{x\}$.  Also, if $x$ is a heavy vertex of $K$ then each father of $x$ is a father of a unique vertex in $K\setminus\{x,y\}$ by Claim~\ref{c1}, where $y$ is the twin of $x$. Thus we can say that each father of a vertex of $K$ is a father of exactly two vertices in $K$.
    
    Let $x_1,\dots,x_{k+1}$ be the vertices of $K$. Suppose without loss of generality that $x_1$ is a non-heavy $k$-vertex. 
    Let $l\in \{1,\dots, k-1\}\setminus m_S(x_1)$ and let $P$ be an alternating $\overline{S}S_l$-path of maximal length containing $x_1$. Since $x_1$ has a unique $l$-father, then $x_1$ is an end vertex of $P$. Set $P=x_1u_1\dots u_r$, $r\geq 1$. Since each father of a vertex of $K$ is a father of two vertices of $K$, then $r\geq 2$, $P\cap\overline{S}\subset K$,  and $u_r$ is a vertex of $K$.  Accordingly, $P$ is an alternating $\overline{S}S_l$-path whose both ends are in $\overline{S}$, a contradiction by Claim~\ref{c5}.     
\end{proof}

Hence, by Brooks' Theorem $G^2[\overline{S}]$ is $k$-colorable. Color the vertices of $G[\overline{S}]$ by $2_1,\dots, 2_{k}$, and so we obtain a $(1^{k-1}, 2^k)$-packing coloring of $G$. 
\end{proof}


\section{Sharpness of the Results and Open Problems}
\label{s:concl}
For Theorem 2.1, we can construct a $k$-degree $0$-saturated graph which is not $(1^{k-1},4)$-packing colorable: take two copies of the complete graph $K_k$ and link one vertex of the first copy to one vertex of the second copy by a path of length two. If we have only $k-1$ colors 1, then (at least) one vertex of each $K_k$ cannot be colored, and since the graph has diameter 4, then the two uncolored vertices cannot be colored both with color 4.

For the sharpness of Theorems~\ref{tsat}, \ref{k-1sat}, and \ref{kdeg}, we construct a counter-example for $(1^{k-1},3^t)$-colorability for $t$-saturated $k$-degree graphs, $1\le t\le k$.

\begin{proposition}
For any $k\ge 3$ and $1\le t\le k$, there exists a $t$-saturated $k$-degree graph that is not $(1^{k-1},3^t)$-packing colorable.
\end{proposition}
\begin{proof}
Let $G_{k,t}$ be the graph constructed in this way: Start from $t+1$ copies of $K_k$ and link a different vertex of each $K_k$ to a vertex of every other $K_k$ in such a way any vertex of the obtained graph has degree $k$ or $k-1$ (i.e., each vertex of a $K_k$ is linked with at most one vertex of any other $K_k$).  Thus $G_{k,t}$ is a $k$-degree graph and it is $t$-saturated. In fact, if $x$ is a vertex of a copy $K_k$ and $x$ is of degree $k$, then $x$ is adjacent to a vertex $y$ of another copy of $K_k$. Thus $x$  has $t$ neighbors of degree $k$ ($t-1$ in the same copy of $K_k$ plus the vertex $y$). See Figure~\ref{fig:ce} for an example with $k,t=5,3$. \\
Hence $G_{k,t}$ is not $(1^{k-1},3^t)$-colorable since using only the $k-1$ colors 1, there will remain one uncolored vertex in each of the $t+1$ copies of $K_k$ and $t$ colors 3 are not enough to color them since $G_{k,t}$ has diameter 3.
\end{proof}

This example, along with Theorem~\ref{tsat} shows that replacing one color 2 with $t$ colors 3 is not always sufficient for $t$-saturated $k$-degree graphs, $1\le t\le k-2$.

\begin{figure}[ht]
\begin{center}
\begin{tikzpicture}[scale=1.3]
\node at (3.6,3) (a1) [circle,draw=black,fill=white,scale=0.5]{};
\node at (4.4,3) (b1) [circle,draw=black,fill=white,scale=0.5]{};
\node at (4.75,3.85) (c1)[circle,draw=black,fill=white,scale=0.5]{};
\node at (4,4.35) (d1) [circle,draw=black,fill=white,scale=0.5]{};
\node at (3.25,3.85) (e1) [circle,draw=black,fill=white,scale=0.5]{};
\draw (a1) -- (b1) -- (c1) -- (d1) -- (e1) -- (a1);
\draw (a1) -- (c1) -- (e1) -- (b1) -- (d1) -- (a1);
\node at (6.6,3) (a2) [circle,draw=black,fill=white,scale=0.5]{};
\node at (7.4,3) (b2) [circle,draw=black,fill=white,scale=0.5]{};
\node at (7.75,3.85) (c2)[circle,draw=black,fill=white,scale=0.5]{};
\node at (7,4.35) (d2) [circle,draw=black,fill=white,scale=0.5]{};
\node at (6.25,3.85) (e2) [circle,draw=black,fill=white,scale=0.5]{};
\draw (a2) -- (b2) -- (c2) -- (d2) -- (e2) -- (a2);
\draw (a2) -- (c2) -- (e2) -- (b2) -- (d2) -- (a2);
\node at (3.6,0) (a3) [circle,draw=black,fill=white,scale=0.5]{};
\node at (4.4,0) (b3) [circle,draw=black,fill=white,scale=0.5]{};
\node at (4.75,0.85) (c3)[circle,draw=black,fill=white,scale=0.5]{};
\node at (4,1.35) (d3) [circle,draw=black,fill=white,scale=0.5]{};
\node at (3.25,0.85) (e3) [circle,draw=black,fill=white,scale=0.5]{};
\draw (a3) -- (b3) -- (c3) -- (d3) -- (e3) -- (a3);
\draw (a3) -- (c3) -- (e3) -- (b3) -- (d3) -- (a3);
\node at (6.6,0) (a4) [circle,draw=black,fill=white,scale=0.5]{};
\node at (7.4,0) (b4) [circle,draw=black,fill=white,scale=0.5]{};
\node at (7.75,0.85) (c4)[circle,draw=black,fill=white,scale=0.5]{};
\node at (7,1.35) (d4) [circle,draw=black,fill=white,scale=0.5]{};
\node at (6.25,0.85) (e4) [circle,draw=black,fill=white,scale=0.5]{};
\draw (a4) -- (b4) -- (c4) -- (d4) -- (e4) -- (a4);
\draw (a4) -- (c4) -- (e4) -- (b4) -- (d4) -- (a4);
\draw (d1) -- (d2);
\draw (e1) -- (e3);
\draw (c2) -- (c4);
\draw (b3) -- (a4);
\draw (b1) -- (e4);
\draw (a2) -- (c3);

\end{tikzpicture}
\end{center}
\caption{\label{fig:ce} The graph $G_{5,3}$ is a  $3$-saturated $5$-degree graph that is not $(1^4,3^3)$-packing colorable.}
\end{figure}
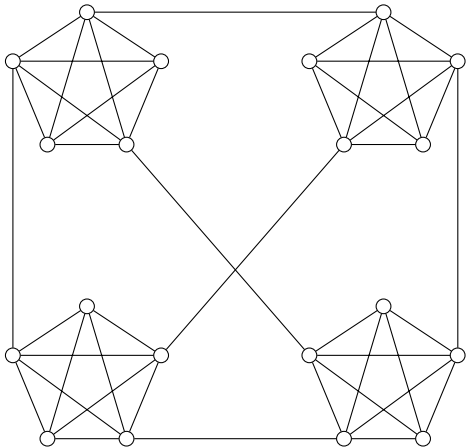

In~\cite{LZZ24}, the authors proved that every subcubic graph is $(1,1,2,2,3)$-packing colorable. Inspired by this result, we propose the following conjecture that refines Theorem~\ref{kdeg}:
\begin{conjecture}
    For any $k\geq 3$, every $k$-degree graph is $(1^{k-1},2^{k-1},3)$-packing colorable.
\end{conjecture}
From our experience of working in the domain of packing coloring of graphs, we also propose these problems:
\begin{problem}
    Prove or disprove that every 0-saturated $k$-degree graph is $(1^{k-2}, 2^k)$-packing colorable.
\end{problem}

\begin{problem}
    Prove or disprove  that every $(k-1)$-saturated $k$-degree graph, is $(1^{k-1}, 2,2)$-packing colorable.
\end{problem}

\begin{problem}
    Prove or disprove that every $k$-degree graph, $k\geq 4$, is $(1^{k-2}, 2^{2k})$-packing colorable. 
\end{problem}

   \end{document}